\documentclass[a4paper,10pt,reqno]{amsart}
\usepackage{latexsym, stmaryrd, mathtools, hyperref}
\usepackage{tikz}
\usetikzlibrary{matrix}
\usepackage{booktabs}

\DeclareMathAlphabet{\mathitbf}{OML}{cmm}{b}{it}

\newcommand{\ZZ}{{\mathbb Z}}    

\newcommand{\sym}{\mathcal{S}}   
\newcommand{\Av}{\mathrm{Av}}    
\newcommand{\A}{\mathcal{A}}     
\newcommand{\B}{\mathcal{B}}     
\newcommand{\C}{\mathcal{C}}     
\newcommand{\D}{\mathcal{D}}     
\newcommand{\R}{\mathcal{R}}     
\newcommand{\Dec}{\mathcal{D}}   
\newcommand{\Inc}{\mathcal{I}}   
\newcommand{\K}{\mathcal{K}}   
\newcommand{\X}{\mathcal{X}}   
\newcommand{\ledot}{<\!\!\!\cdot\,\,}

\newcommand{\V}{\ensuremath{\mathsf{V}}}
\newcommand{\Lt}{\rotatebox[origin=c]{270}{\V}}
\newcommand{\Gt}{\rotatebox[origin=c]{90}{\V}}
\newcommand{\Caret}{\rotatebox[origin=c]{180}{\V}}

\theoremstyle{plain}
\newtheorem{theorem}{Theorem}
\newtheorem*{theorem*}{Theorem}
\newtheorem{corollary}[theorem]{Corollary}
\newtheorem*{corollary*}{Corollary}
\newtheorem{lemma}[theorem]{Lemma}
\newtheorem*{lemma*}{Lemma}
\newtheorem{proposition}[theorem]{Proposition}
\newtheorem*{proposition*}{Proposition}

\newtheorem*{conjecture*}{Conjecture}

\theoremstyle{definition}

\newtheorem*{definition*}{Definition}

\newtheorem*{example*}{Example}

\newtheorem*{problem*}{Problem}

\theoremstyle{remark}

\newtheorem*{remark*}{Remark}

\DeclareMathOperator{\Cl}{\mathrm{Cl}}

\DeclareMathOperator{\shadow}{\ensuremath{\Delta}}

\title{Isomorphisms between pattern classes}

\author[Albert]{M. H. Albert}
\author[Atkinson]{M. D. Atkinson}
\author[Claesson]{Anders Claesson}

\address{M. H. Albert and M. D. Atkinson:
Department of Computer Science,
University of Otago,
Dunedin, New Zealand
}
\address{A. Claesson: 
Department of Computer Science,
University of Strathclyde, Glasgow, UK }

\begin{document}

\begin{abstract}
Isomorphisms $\phi:\A\longrightarrow\B$ between pattern classes are considered.  It is shown that, if $\phi$ is not a symmetry of the entire set of permutations, then, to within symmetry, $\A$ is a subset of one a small set of  pattern classes whose structure, including their enumeration, is determined.
\end{abstract}

\maketitle
\thispagestyle{empty}


\section{Introduction}\label{sec1}

The set $\sym$ of all (finite) permutations is a partially ordered set under the pattern containment order.  It has been known since the work of Simion and Schmidt \cite{simion:restricted-perm:} that this poset has eight poset automorphisms induced by the symmetries of the square which act in a natural way on the diagrams of permutations.  Throughout this paper we shall simply call these automorphisms \emph{symmetries}.  It was tacitly assumed for many years that there were no further automorphisms of this poset but this fact was not explicitly proved until Smith's work \cite{smith:permutation-rec:} on permutation reconstruction.

The pattern containment order is studied almost exclusively by the investigation of down-sets (lower ideals) and these are called \emph{pattern classes}.  
The image of a pattern class under one of the eight symmetries is again a pattern class having the same  order-theoretic properties;  in systematic investigations this can be very  useful in limiting the number of cases that have to be considered.  In particular two pattern classes connected in this way have the same enumeration, a property which is known as \emph{Wilf-equivalence}.  One interesting aspect of the theory of pattern classes is that Wilf-equivalent classes may not be related by a symmetry.  Even more strikingly two Wilf-equivalent pattern classes are usually not even isomorphic as ordered sets. This leads to a natural question:
\begin{quote}
{\it
What order-preserving isomorphisms are there between pattern classes beyond the restrictions of the eight symmetries?
}
\end{quote}

In this paper we shall answer this question.  Our answer is complete in the following sense:

\begin{enumerate}
\item
We do not count two isomorphisms that differ within symmetries as being distinct.  Thus, if $\phi:\A\longrightarrow\B$ is an isomorphism and $\mu$ and $\nu$ are symmetries then the isomorphisms $\phi$ and  $\mu^{-1}\phi\nu:\A\mu\longrightarrow\B\nu$ are essentially the same.\footnote{Note that we write the action of functions on sets in line and on the right, while we will denote the action of a function on an element by exponential notation.}
\item
If $\phi:\A\longrightarrow\B$ is an isomorphism and there are larger classes $\A\subseteq\A_0$ and $\B\subseteq\B_0$ with an isomorphism $\psi$ between them for which $\psi |_{\A}=\phi$ then we do not regard $\phi$ as the ``real'' isomorphism.  In other words we shall be chiefly interested in \emph{maximal isomorphisms} $\phi:\A\longrightarrow\B$ that cannot be extended to a larger class.
\end{enumerate}

To expand on the second point, note that the union of a chain (ordered by inclusion) of isomorphisms is also an isomorphism, so each isomorphism is in fact contained in a maximal one -- but in fact we will see as a consequence of Proposition \ref{prop:finite} below that there are even more restrictive conditions which make it possible to carry out the desired classification.

In the next section we shall give the necessary notation and background to understand our basic method.  Then in Section \ref{results} we construct the maximal isomorphisms $\phi:\A\longrightarrow\B$ as well as describe the structure and enumeration of the classes $\A$ on which they are defined.

\section{Preparatory background}
\label{prepsection}
Let $P$ be a poset and let $x < y$ in $P$. We say that $y$ \emph{covers}
$x$, denoted $x \ledot y$, if no $z\in P$ satisfies $x < z < y$. We also define the \emph{shadow} and \emph{downset} (down closure) of $y\in P$
by
\[
\shadow y = \{ x\in P : x\ledot y \}
\quad\text{and}\quad
\Cl(y) = \{ x\in P : x \leq y \}
\]
A bijection  $f:P\to Q$ between two posets $P$ and $Q$ is an  \emph{isomorphism} if $f(x)<f(y)$ if and only if $x<y$. If, in addition,  $P=Q$ then $f$ is
an \emph{automorphism}.  All the posets we will be considering are lower ideals of $\sym$. All these posets are \emph{ranked} -- for every permutation $\pi$ the lengths of the maximal chains with maximum element $\pi$ are all (finite and) the same. In particular, the rank of $\pi \in \sym$ is equal to the length of $\pi$ (number of letters in $\pi$). Any bijection $f$ between ranked posets is an isomorphism exactly when  $f(x)\ledot f(y)$ if and only if $x\ledot y$.  Since the covering relations of an ordered set are precisely the relations $x<y$ where $x\in\shadow(y)$ we have

\begin{lemma}\label{lemma:shadow-and-downset}
  A bijection $f:P\to Q$ between ranked posets is an isomorphism if and only if
  $(\shadow y)^f = \shadow(y^f)$ for all $y\in P$.
\end{lemma}

Smith~\cite{smith:permutation-rec:} has shown
that no pairs of permutations of length at least five share the same
shadow.  This is a key step in her proof that the automorphism group of symmetries of $\sym$ is the dihedral group of order eight; it contains reverse ($r$), complement ($c$) and inverse
($i$). If we refine her shadow result slightly by spelling out what happens
for permutations of fewer than five elements we arrive at the following
proposition.

\begin{proposition}\label{prop:smith}
  Let $\sigma$ and $\tau$ be permutations. Then
  \begin{align*}
    \shadow \sigma  &= \shadow \tau
  \intertext{if and only if one of the following holds}
    \{\sigma,\tau\} &= \{12,21\}; \\
    \{\sigma,\tau\} &\subseteq \{132,213,231,312\}; \\
    \{\sigma,\tau\} &= \{2413,3142\}; \\
    \sigma &= \tau.
  \end{align*}
\end{proposition}


As mentioned in Section \ref{sec1} down-sets of $\sym$ are called pattern classes.   We shall say that two pattern classes are
\emph{ isomorphic} if they are  isomorphic as posets. In
this paper we shall study  isomorphisms between classes.  From
Lemma~\ref{lemma:shadow-and-downset} and Proposition~\ref{prop:smith}
we get the following result.

\begin{proposition}\label{prop:finite}
  Let $\A$ and $\B$ be pattern classes. Then any  isomorphism
  $f:\A\to\B$ is determined by its restriction $f|_{\R}$ to the finite
  pattern class
  \[\R = \Cl\{2413,3142\} = \{1,12,21,132,213,231,312,2413,3142\}.
  \]
\end{proposition}

\begin{proof}
  Suppose that both $f:\A\to\B$ and $g:\A\to\B$ are 
  isomorphisms. Suppose further that $f|_\R = g|_\R$. For the sake of contradiction suppose that there exists $\pi \in \A$ with $\pi^f \neq \pi^g$. Among all such permutations $\pi$, choose one (also called $\pi$) of minimum length. Then 
  \begin{align*}
    \shadow (\pi^f)
    &= (\shadow\pi)^f  && \text{by Lemma~\ref{lemma:shadow-and-downset}}  \\
    &= (\shadow\pi)^g  && \text{by minimality} \\
    &= \shadow (\pi^g) && \text{by Lemma~\ref{lemma:shadow-and-downset}}.
  \end{align*}
  Using Proposition~\ref{prop:smith} it follows that $\pi^f = \pi^g$ and so we have a contradiction.
\end{proof}

%
%
%

By Proposition~\ref{prop:finite}, any  isomorphism $f:\A\to\B$ is
uniquely determined by its restriction to $\R=\Cl\{2413,3142\}$. Since
 isomorphisms are length preserving, it must further hold that
$1^f=1$, $\{12,21\}^f = \{12,21\}$, $\{132,213,231,312\}^f =
\{132,213,231,312\}$ and $\{2413,3142\}^f = \{2413,3142\}$. 

The reader may well wonder ``what if $\A$ does not contain $\R$?'' There are two ways to answer that question. First, and simplest, is to interpret $f |_{\R}$ as $f |_{\R \cap \A}$. Alternatively it is easy to see that $f$ could be extended to an isomorphism between $\A \cup \R$ and $\B \cup \R$ -- by taking the union with any length-preserving bijection between $\R \setminus \A$ and $\R \setminus \B$.


This motivates the following approach to the original question: begin with an  isomorphism $h:\R\to \R$; then use Lemma~\ref{lemma:shadow-and-downset} repeatedly to define extensions to larger pattern classes obtaining, in the limit, the \emph{maximal} class $\A$ containing $\R$ and image class $\B$ such that $h$ extends to an  isomorphism $f:\A\to \B$.   To this end, let $\bot$ be
an arbitrary symbol standing for ``undefined''. We first define a function 
$g:\sym\to\sym\cup\{\bot\}$ recursively:
\[
\pi^g =
\begin{cases}
  \pi^h  & \text{if $\pi\in \R$}, \\
  \sigma & \text{if $\pi\not\in\R$ but $(\shadow\pi)^g = \shadow \sigma$ for some $\sigma\in\sym$},\\
  \bot   & \text{otherwise}.
\end{cases}
\]
Note that the second condition implicitly requires that $\theta^g \neq \bot$ for any $\theta \in \shadow \pi$. Now set $\A=\{\pi\in\sym : \pi^g \neq \bot \}$ and $\B = \A^f$, with $f =
g|_{\A}$.  Then $(f,\A,\B)$ is the required triple.

This recursive definition gives rise to an infinite process that converges to the triple $(f,\A,\B)$ -- and to the \emph{basis}, $B$, of $\A$ (those $\beta \in \sym \setminus \A$ for which $\shadow \beta \subseteq \A$). We accumulate the elements of $\A$ and of its basis in order of length. Given a permutation $\alpha$ not presently in $\A$ whose lower covers do belong to $\A$,  we apply $f$ to $\shadow\alpha$ and, if
the result is a shadow of some permutation $\beta$,  we add $\alpha$ to $\A$ and set $\alpha^f =
\beta$; otherwise we add $\alpha$ to the basis $B$ of the class $\A$.  

Considering only the basis of $\A$
there is no {\em a priori} guarantee that this process will terminate, since it is conceivable that the class $\A$ might have an infinite basis.  However, in practice, we always reach a situation where no new basis elements are found in small lengths 7, 8, 9.  At this stage we have a pattern class $\A$ (at this point finite!) on which is defined an isomorphism $f$ and we have a basis $B$ of some class that contains $\A$.  We hope that $\Av(B)$ is the maximal class on which an isomorphism extending $f$ exists. Since $\Av(B)$ necessarily contains every such class it now suffices to prove that $f$ can be defined to an isomorphism on the whole of $\Av(B)$.  This will prove that $\A=\Av(B)$.  It turns out that this can be achieved by a structural analysis of $\Av(B)$. We carry out this analysis in the next section. For any particular map $h$, the detailed calculation of this structure will be largely omitted since it is relatively routine\footnote{We actually used \emph{PermLab} \cite{PermLab1.0} to compute this structure but it is easily within reach of hand calculation}.  We shall however, in each case, display a map defined on $\Av(B)$ that extends $h$ and prove that it is isomorphism.

Before continuing with the consideration of maximal isomorphisms we give a non-maximal isomorphism that features several times in the discussion.  The class $\Av(132, 312)$ is well-known to consist of permutations which are the union of an increasing sequence with a decreasing sequence where the increasing terms are all greater than the decreasing terms.  The diagrams  of such permutations have the form


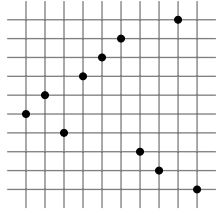
\[
\begin{tikzpicture}[scale=0.25, baseline=(current bounding box.center)]
  \draw[gray, line width=0.5pt] (0.01,0.01) grid (10.99, 10.99);
  \foreach \x/\y in {1/5,2/6,3/4,4/7,5/8,6/9,7/3,8/2,9/10,10/1} \filldraw (\x,\y) circle (5.2pt);
\end{tikzpicture}
\]
 
 and, for obvious reasons, we name this class $\Lt$. There are 3 further `wedge' classes related to it by symmetries which we call $\Gt$, $\V$ and $\Caret$.  The class $\Lt$ has an obvious automorphism induced by the complement symmetry but, as we shall now see, there is another automorphism not induced by any symmetry.
 
 Let $\pi$ be a permutation of length $n$ in $\Lt$ and define the
word $\pi^\omega = c_2c_3\dots c_n$ in $\{a,b\}^{n-1}$ by
\[ c_i =
\begin{cases}
  a & \text{if $\pi(i) > \pi(1)$}, \\
  b & \text{if $\pi(i) < \pi(1)$},
\end{cases}
\]
It is easy to see that $\omega$ is a
bijection and thus has an inverse $\omega^{-1}$. We can now define a bijection on $\Lt$ by
\[\xi = \omega r \omega^{-1}
\]
Here $r$ denotes reversal (of words) and composition is left to right. As an
example, $45367821^\omega = abaaabb$; the reverse of this string is
$bbaaaba = 43256718^\omega$; and thus $45367821^\xi = 43256718$:
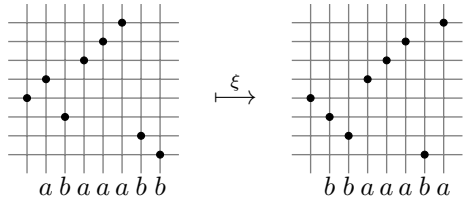
\[
\begin{tikzpicture}[scale=0.25, baseline=(current bounding box.center)]
  \draw[gray, line width=0.5pt] (0.01,0.01) grid (8.99, 8.99);
  \foreach \x/\y in {1/4,2/5,3/3,4/6,5/7,6/8,7/2,8/1} \filldraw (\x,\y) circle (5.2pt);
  \foreach \x in {2,4,5,6} \node[text height=2.1ex] at (\x,-0.5) {$a$};
  \foreach \x in {3,7,8} \node[text height=2.1ex] at (\x,-0.5) {$b$};
\end{tikzpicture}\quad\stackrel{\xi}{\longmapsto}\quad
\begin{tikzpicture}[scale=0.25, baseline=(current bounding box.center)]
  \draw[gray, line width=0.5pt] (0.01,0.01) grid (8.99, 8.99);
  \foreach \x/\y in {1/4,2/3,3/2,4/5,5/6,6/7,7/1,8/8} \filldraw (\x,\y) circle (5.2pt);
  \foreach \x in {4,5,6,8} \node[text height=2.1ex] at (\x,-0.5) {$a$};
  \foreach \x in {2,3,7} \node[text height=2.1ex] at (\x,-0.5) {$b$};
\end{tikzpicture}
\]

Since $\xi$ is a fairly unusual bijection we shall call it the \emph{exotic map} on $\Lt$.

\begin{proposition}
The exotic map is an automorphism of $\Lt$.
\end{proposition}
\begin{proof}
The set $\K$ of words over the alphabet $\{a,b\}$ is an ordered set  under the subsequence ordering.  Its covering relations $\lambda<\mu$ are those relations in which $\lambda$ is obtained from $\mu$ by deleting a single letter.  We shall show that $\omega$ is an isomorphism of ordered sets by verifying that it maps the set of covering relations of $\Lt$ onto the set of covering relations of $\K$.

A covering relation $\alpha<\beta$ of $\Lt$ is a pair of permutations in which $\alpha$ is obtained from $\beta$ by deleting a term of $\beta$ (and relabeling).  If this deleted term is not the first term then the corresponding words of $\K$  are related in that one word is obtained from the other by deleting the letter that corresponds to the deleted term.  But, if the deleted term is the first term then the resulting permutation is the same as that obtained when we delete the second term. The converse is equally clear.

Therefore the exotic map is a composition of  isomorphisms $\omega r \omega^{-1}$ on ordered sets and so is also an isomorphism.
\end{proof}

Finally in this section we recall some standard terminology.  An \emph{interval} in a permutation $\pi$ is a contiguous subsequence of $\pi$ whose terms form a set of consecutive values.  If the only intervals of $\pi$ are itself and singletons then $\pi$ is said to be \emph{simple}.  Notation such as $\pi=\sigma[\alpha_1,\ldots,\alpha_r]$ signifies that $\pi$ can be represented as a juxtaposition of intervals $\pi=\alpha_1\cdots\alpha_r$ where the pattern formed by the $\alpha_i$ defines the permutation $\sigma$.  In such a case we say that $\pi$ is obtained from $\sigma$ by \emph{inflating} its points into intervals $\alpha_i$.  In the special cases $\sigma=12$ or $\sigma=21$ then we say that $\pi$ is decomposable or skew-decomposable respectively.


\section{The maximal isomorphisms}
\label{results}
 We will classify all triples
\[(f,\A,\B),
\]
where $f:\A\to \B$ is an  isomorphism and $\A$ and $\B$ are
maximal classes with this property. This classification will be ``up
to symmetry'' as described in Section \ref{sec1}. In other words, we consider $(f,\A,\B)$ and
$(f',\A',\B')$ to be essentially the same if there are symmetries
$\mu$ and $\nu$ of $\sym$ such that the following diagram commutes:
\[
\begin{tikzpicture}[node distance=1.5cm, auto]
  \node (A) {$\A$};
  \node (B) [right of=A] {$\B$};
  \node (A') [below of=A] {$\A'$};
  \node (B') [below of=B] {$\B'$};
  \draw[->] (A) to node [swap] {$\mu$} (A');
  \draw[->] (B) to node {$\nu$} (B');
  \draw[->] (A) to node {$f$} (B);
  \draw[->] (A') to node [swap] {$f'$} (B');
\end{tikzpicture}
\]

Although it is not at all clear {\em a priori} that, when $(f,\A,\B)$ is a maximal isomorphism, there will be a \emph{symmetry} mapping $\A$ onto $\B$ this will indeed prove to be the case.  So, in fact, every maximal isomorphism is equivalent to some maximal automorphism $(f,\A,\A)$.  Remembering that the maps $f$ are built up from an initial isomorphism $h: \R \to \R$,  we shall take some care (available only after the fact) in the choice of representative maps $h$ on $\R$ to allow this result to emerge naturally.

Since we could arbitrarily permute the permutations of length two, three, and four in $\R$ (preserving length of course) there are apparently 96 maps $h$ to consider. Some routine calculations establish that there are only 6 equivalence classes of such maps under the symmetry operations given above.  We shall use the  representatives given in Table \ref{6maps} (recall that an isomorphism necessarily maps $1$ to $1$).


\begin{table}
\caption{6 representative bijections on $\R$}
\label{6maps}
\begin{center}
\begin{tabular}{lllllllll}
\toprule
&\multicolumn{8}{c}{Elements of $\R$ and their images}\\
Bijection&$12$&$21$&$132$&$213$&$231$&$312$&$2413$&$3142$\\
\midrule
$h^{(1)}$&$12$&$21$&132&213&231&312&2413&3142\\
$h^{(2)}$&$12$&$21$&132&213&231&312&3142&2413\\
$h^{(3)}$&$12$&$21$&132&231&213&312&2413&3142\\
$h^{(4)}$&$12$&$21$&132&231&213&312&3142&2413\\
$h^{(5)}$&$12$&$21$&231&312&132&213&2413&3142\\
$h^{(6)}$&$12$&$21$&231&312&213&132&2413&3142\\
\bottomrule
\end{tabular}
\end{center}
\end{table}


The procedure given above for finding the maximal extension of each map $h:\R\longrightarrow\R$ to an isomorphism $f:\A\longrightarrow\B$ results in the basis $B$ for $\A$ given in Table \ref{bases}.  Recall that, until we have exhibited an isomorphism on $\Av(B)$ that extends $h$, we only know that $\A\subseteq \Av(B)$.  The necessary verifications are carried out below.  In addition we shall give the enumerations of the various classes.

For easy reference let
\[(f^{(i)},\A^{(i)},\B^{(i)}),
\]
where $\B^{(i)} = (\A^{(i)})^{f^{(i)}}$, be the maximal isomorphism that extends the $i$th bijection on $\R$ in Table \ref{6maps}. For instance, $f^{(3)}$ is the maximal  isomorphism
defined by fixing $132$ and $312$ while interchanging $213$ and $231$;
and
\[\A^{(3)} = \Av(2143, 2431, 3124, 3412, 23514, 25134, 31452, 35214, 41532, 43152)\]

\begin{table}
\caption{Bases for classes $\A^{(1)},\ldots,\A^{(6)}$}
\label{bases}
\begin{tabular}{ll}
\toprule
Class&\multicolumn{1}{c}{Basis}\\
\midrule
$\A^{(1)}$&$\emptyset$\\
\midrule
$\A^{(2)}$&
	\begin{tabular}{llllll}
	23514&24513&25134&25143&25314&25413\\
	31452&31542&32514&34152&35124&35214\\
	41253&41352&41523&41532&42153&43152\\
	241635&315264&462513&536142&&\\
	\end{tabular}\\
\midrule
$\A^{(3)}$&
	\begin{tabular}{llllll}
	2143&2431&3124&3412&&\\
	25134&23514&31452&35214&41532&43152\\
	\end{tabular}\\	
\midrule
$\A^{(4)}$&
	\begin{tabular}{llllll}
	2143&2431&3124&3412&&\\
	23514&25134&31452&35214&41532&43152\\
	\end{tabular}\\	
\midrule
$\A^{(5)}$&
	\begin{tabular}{llllll}
	1324&4231&&&\\
	14253&21354&21453&21534&21543&23154\\
	23514&24153&24513&25134&25143&25413\\
	31254&32452&31524&31542&32154&32514\\
	32541&34125&34152&34512&35124&35142\\
	35214&35412&41253&41523&41532&42153\\
	42513&43152&43512&45123&45132&45213\\
	45312&52143&&&&\\
	\end{tabular}\\	
\midrule
$\A^{(6)}$&
	\begin{tabular}{llllll}
	1324&4231&&&\\
	14253&21354&21453&21534&21543&23154\\
	23514&24153&24513&25134&25143&25413\\
	31254&32452&31524&31542&32154&32514\\
	32541&34125&34152&34512&35124&35142\\
	35214&35412&41253&41523&41532&42153\\
	42513&43152&43512&45123&45132&45213\\
	45312&52143&&&&\\
	25314&41352&&&&\\
	\end{tabular}\\	
\bottomrule
\end{tabular}
\end{table}

\subsection*{The triple $(f^{(1)},\A^{(1)},\B^{(1)})$} 
The mapping
$f^{(1)}$ fixes all permutations of $\R$ and is thus the
identity mapping; $f^{(1)}$ can be seen as the representative for any
of the symmetries  of the full pattern containment order. Moreover, the basis for $\A^{(1)}$ is
empty and so $\A^{(1)} = \B^{(1)} = \sym$.

\subsection*{The triple $(f^{(2)},\A^{(2)},\B^{(2)})$}


Now let $\A$ stand for the class defined by the putative basis elements of  $\A^{(2)}$.  The class $\A$ is related
to the separable permutations\footnote{The class with basis $\{2413, 3142\}$, alternatively described as the smallest non-empty class $\X$ such that both $12[\alpha, \beta]$ and $21[\alpha, \beta]$ belong to $\X$ whenever $\alpha, \beta \in \X$.}. It also includes two simple permutations of
length four, and four of length five:
\[T = \{2413, 3142, 24153, 31524, 35142, 42513\}.
\]
None of the permutations in $T$ can be non-trivially inflated. Thus
$\A$ is the $\oplus$ and $\ominus$ completion of the finite class
$\Cl(T)$. In other words, $\A$ is defined by the following system of
set-equations:
\begin{align*}
  \A &= \C \cup \D \cup U,  \\
  \C &= (\D\cup U) \oplus \A, \\
  \D &= (\C\cup U) \ominus \A,
\end{align*}
where $U = \{1\}\cup T$. In these equations $\C$ and $\D$ stand, respectively, for the set of sum-decomposable and skew-decomposable permutations of $\A$.  In this representation of $\A$
all the unions are disjoint, a property we will now use. Let
$g : U \to U$ be defined by $2413^g = 3142$, $3142^g = 2413$ and
$\pi^g = \pi$ for $\pi \in U\setminus \{2413, 3142\}$. 

Now extend the definition of $g$ to the whole of $\A$ by defining it recursively on sums and skew sums:

\begin{align*}
  \pi^g &=
  \begin{cases}
    \sigma^g\oplus\tau^g   & \text{if $\pi\in\C$ and $\pi = \sigma\oplus\tau$}, \\
    \tau^g\ominus\sigma^g  & \text{if $\pi\in\D$ and $\pi = \sigma\ominus\tau$}.\\
  \end{cases}
\end{align*}

Since $1^g=1$ it is clear that $g$ fixes all permutations of length up to 3.  Since also $g$ exchanges $2413$ and $3142$ we have $g|_{\R} = h^{(2)}|_{\R}$.  It remains only to prove that $g$ is order preserving.
By
Lemma~\ref{lemma:shadow-and-downset} it suffices to prove that
\begin{equation}\label{eq:2}
  (\shadow \pi)^g = \shadow(\pi^g)
\end{equation}
for all $\pi\in \A$.  The proof will again use contradiction starting from the assumption that $\pi$ is a counterexample to the above of minimal length.


That
\eqref{eq:2} holds for $\pi\in U\cup\R$ is easy to check by direct
calculations, so our minimal counterexample must lie in $\A \setminus (U \cup \R)$. We first make a simple observation: for any permutations
$\sigma$ and $\tau$,
\begin{align}
  \shadow (\sigma\oplus\tau)
  &= (\shadow\sigma) \oplus \tau \,\cup\, \sigma\oplus(\shadow\tau);
  \label{eq:shadow-plus}\\
  \shadow (\sigma\ominus\tau)
  &= (\shadow\sigma) \ominus \tau \,\cup\, \sigma\ominus(\shadow\tau).
  \label{eq:shadow-minus}
\end{align}
Assume that $\pi=\sigma\oplus\tau$ with $\sigma \in \D\cup U$ and
$\tau \in \A$. Then
\begin{align*}
  \shadow(\pi^g)
  &= \shadow(\sigma^g \oplus \tau^g)
  && \text{by definition of $g$} \\
  &= (\shadow(\sigma^g)) \oplus \tau^g \,\cup\, \sigma^g\oplus(\shadow(\tau^g))
  && \text{by~\eqref{eq:shadow-plus}} \\
  &= (\shadow \sigma)^g \oplus \tau^g \,\cup\, \sigma^g\oplus(\shadow\tau)^g
  && \text{by minimality} \\
  &= ((\shadow \sigma) \oplus \tau)^g \,\cup\, (\sigma\oplus(\shadow\tau))^g
  && \text{by definition of $g$} \\
  &= (\shadow (\sigma \oplus \tau))^g
  && \text{by~\eqref{eq:shadow-plus}} \\
  &= (\shadow \pi)^g.
\end{align*}
So we would have a contradiction in this case, and similarly if $\pi$ were skew-decomposable.


In summary, the isomorphism $g$ behaves  like the identity except that it exchanges intervals $2413$ and $3142$ wherever they occur. It is easy to see that
$g$ is its own inverse. In particular, $g$ is an automorphism and
hence $\B^{(2)} = \A^{(2)}$.

The system of set-equations above directly translates to a system of
equations for the corresponding generating functions
\begin{align*}
  A(x) &= C(x) + D(x) + U(x), \\
  C(x) &= (D(x) + U(x))A(x),  \\
  D(x) &= (C(x) + U(x))A(x),
\end{align*}
in which $U(x) = 1 + 2x^4 + 4x^5$. Solving for $A(x)$ we find that
\[A(x) = \frac{\sqrt{p(x)}-q(x)-2}{q(x)-\sqrt{p(x)}-2},
\]
where
\begin{align*}
  p(x) &= 1 - 6x + x^2 - 12x^4 - 20x^5 + 8x^6 + 4x^8 + 16x^9 + 16x^{10},\\
  q(x) &= -1 + x + 2x^4 + 4x^5.
\end{align*}
The first few values of $\{|\A_n|\}_{n\geq 1}$ are
\[ 1, 2, 6, 24, 102, 446, 2054, 9818, 48218, 241686, 1231214, 6356050, 33178450,\dots
\]

 \subsection*{The triple $(f^{(3)},\A^{(3)},\B^{(3)})$} 
 Let $\A$ denote the class defined by the avoidance conditions in the $\A^{(3)}$ row of Table \ref{bases}.  In this case  $\A$
is the $\Gt$ class with its apex inflated to 2413, 3142 or any
permutation in $\Lt$:
\begin{equation}\label{eq:A3}
\A = \big\{ \sigma[1,\dots,1,\tau] : \sigma\in\Gt, \tau\in\Lt \cup \{2413,3142\} \big\}
\end{equation}

We need to define an isomorphism on $\A$ that extends $h^{(3)}$.  We shall use the exotic map $\xi$ defined in Section \ref{prepsection} as an automorphism of $\Lt$ and we extend it to $\Lt \cup \{2413,3142\}$ by having it fix $2413$ and $3142$.  Now we define a map $g$ on $\A$ by defining it on a general element $\pi=\sigma[1,\dots,1,\tau]$, with $\sigma\in\Gt$ and $\tau\in\Lt \cup
\{2413,3142\}$ as
\[\pi^g = \sigma[1,\dots,1,\tau^\xi]\]

%
The permutations $1$, $12$, $21$, $132$
and $312$ are in $\Gt$ and so are fixed by $g$. Since the permutations
$2413$ and $3142$ are fixed by $\xi$ they are also fixed by $g$. For
the remaining two permutations in $\R = \Cl\{2413,3142\}$ we have
\begin{align*}
  213^g &= 1[213^\xi] = 1[231] = 231; \\
  231^g &= 1[231^\xi] = 1[213] = 213.
\end{align*}
Thus $g$ agrees with $h^{(3)}$ on $\R$.  To prove that $g$ is an isomorphism we have to show that $(\shadow \pi)^g = \shadow(\pi^g)$ for all
$\pi$ in $\A$. To this end we note that, for $\pi=
\sigma[1,\dots,1,\tau]\in \A$, we have
\begin{equation*}
  \shadow\pi \,=\,
  \big\{\, \rho[1,\dots,1,\tau] : \rho\in \shadow \sigma \,\big\} \,\cup\,
  \big\{\, \sigma[1,\dots,1,\rho] : \rho\in \shadow \tau \,\big\}.
\end{equation*}
Thus
\begin{align*}
  \shadow(\pi^g)
  &= \shadow(\rho[1,\dots,1,\tau^\xi]) \\
  &=
  \big\{\, \rho[1,\dots,1,\tau^\xi] : \rho\in \shadow \sigma \,\big\} \,\cup\,
  \big\{\, \sigma[1,\dots,1,\rho] : \rho\in \shadow(\tau^\xi) \,\big\} \\
  &=
  \big\{\, \rho[1,\dots,1,\tau^\xi] : \rho\in \shadow \sigma \,\big\} \,\cup\,
  \big\{\, \sigma[1,\dots,1,\rho^\xi] : \rho\in \shadow\tau \,\big\} \\
  &=\big(\shadow(\sigma[1,\dots,1,\rho])\big)^g \\
  &=(\shadow\pi)^g.
\end{align*}
Thus $\A=\A^{(3)}$ as required. 
Because the exotic map $\xi$ is an involution, the map $f$ is also an
involution; in particular, $\B^{(3)} = \A^{(3)}$.

We can modify \eqref{eq:A3} a little to make the representation unique. We have
\[\A
= \big\{\, \sigma[1,\dots,1,\tau] \,:\,
\sigma\in\Gt,\, \tau\in \big(\Lt \setminus(\Inc \cup \Dec)\big) \cup \{1,2413,3142\} \,\big\},
\]
where $\Inc = \{1,12,123,1234,\dots\}$ and $\Dec = \{1,21,321,4321,\dots\}$. Thus the generating function $A(x)$ of $\A$ is
\begin{align*}
  A(x)
  &= \frac{x}{1-2x}\left(\frac{x}{1-2x} + 2x^4 - \frac{2x^2}{1-x}\right)\frac{1}{x}\\
  &= \frac{2}{1-x} - \frac{11}{8 \, {\left(1 - 2x \right)}} +
  \frac{1}{2 \,{\left(1 - 2x\right)}^{2}} - x^{3} - \frac{1}{2} \,
  x^{2} - \frac{1}{4} \, x - \frac{9}{8}.
\end{align*}
From this partial fraction decomposition it is routine to derive a
closed formula for $a_n = |\A_n|$, the number of permutations of
length $n$ in $\A$. We find that
\[a_n = 2 - 11\cdot 2^{n-3} + (n+1)\,2^{n-1} = 2 + (4n-7)\,2^{n-3}
\]
for $n\geq 4$. The first few values of $\{a_n\}_{n\geq 1}$ are
\[
1, 2, 6, 20, 54, 138, 338, 802, 1858, 4226, 9474, 20994, 46082, 100354, 217090, \dots
\]

\subsection*{The triple $(f^{(4)},\A^{(4)},\B^{(4)})$} 
This case is almost the same as the previous case.  The only difference is that we must extend the exotic map to interchange (rather than fix) $2413$ and $3142$.  Otherwise the calculations are the same.  Notice that $\A^{(3)}=\A^{(4)}$ and so this class has two inequivalent maximal isomorphisms.  Their product is also an isomorphism of course but it is not a maximal isomorphism.  Clearly the product interchanges $2413$ and $3142$ whenever they occur as intervals and leaves other points unaltered.  It is therefore the isomorphism that extends to the larger class $\A^{(2)}$.  The full automorphism group of $\A^{(3)}$ is thus $\langle c,
f^{(2)}, f^{(3)}\rangle$ and it is easy to see that this group is isomorphic to
$\ZZ_2\times\ZZ_2\times\ZZ_2$.

%

\subsection*{The triple $(f^{(5)},\A^{(5)},\B^{(5)})$}
Let $\A$ be the class defined by the basis in the $\A^{(5)}$ row of Table \ref{bases}. The class $\A$
consists of the four wedge classes together with four extra permutations
of length four, and two of length five:
\begin{equation*}
  \A = \V \cup \Gt \cup \Caret \cup \Lt \cup \{2143,2413,3142,3412,25314,41352\}.
\end{equation*}

An isomorphism $g$ that extends $h^{(5)}$ may be defined as follows.  It will map the wedge classes as indicated in
\[
\begin{tikzpicture}[node distance=2cm, auto]
  \node (V) {$\V$};
  \node (G) [right of=V] {$\Gt$};
  \node (L) [right of=G] {$\Lt$};
  \node (C) [right of=L] {$\Caret$};
  \path (V) edge [loop right] node {$f_1$} (V);
  \draw[->] (G.20) to node {$f_2$} (L.160);
  \draw[->] (L.200) to node {$f_4$} (G.340);
  \path (C) edge [loop left] node {$f_3$} (C);
\end{tikzpicture}
\]
where $f_1 = i \xi i$, $f_2 = r\xi c$, $f_3 = ci\xi ic$, $f_4 = c\xi
r$, and $\xi$ denotes the exotic map (composition of maps is left to right).  It is routine to check that these maps are consistent (e.g. that permutations in $\V\cap\Gt$ are mapped in the same way by $f_1$ and $f_2$).  Furthermore we define $g$ to exchange $2143$ and $3412$ and to fix four permutations $2413$, $3142$, $25314$ and $41352$. 
Since each of $f_1, f_2, f_3, f_4$ are themselves isomorphisms it is clear that $g$ is an isomorphism.   It is straightforward to verify that $g$ extends $h^{(5)}$.

 We further note that the map
$f$ is its own inverse. In particular, it is an automorphism and so
$\B = \A$. Using inclusion-exclusion it is easy to show that
\[|\A_n| = 2^{n+1}-4n+2
\]
for $n\geq 5$. The first few values of $\{|\A_n|\}_{n\geq 1}$ are
\[ 1,2,6,22,48,106,230,482,990,2010,4054,8146,16334,32714,65478,131010,\dots
\]

\subsection*{The triple $(f^{(6)},\A^{(6)},\B^{(6)})$} Let $\A$ be the class defined by the basis in the $\A^{(6)}$ row of Table \ref{bases}.  The class $\A$
is like $\A^{(5)}$ but without the two non-wedge permutations of length $5$
in that class:
\[\A = \V \cup \Gt \cup \Caret \cup \Lt \cup \{2143,2413,3142,3412\}.
\]
In this case an isomorphism $g$ that extends $h^{(6)}$ may be defined as follows.  It maps the wedge classes amongst themselves as shown

\[
\begin{tikzpicture}[node distance=1.5cm, auto]
  \node (V) {$\V$};
  \node (G) [right of=V] {$\Gt$};
  \node (C) [below of=G] {$\Caret$};
  \node (L) [below of=V] {$\Lt$};
  \draw[->] (V) to node {$f_1$} (G);
  \draw[->] (G) to node {$f_2$} (C);
  \draw[->] (C) to node {$f_3$} (L);
  \draw[->] (L) to node {$f_4$} (V);
 \end{tikzpicture}
\]
where $f_1 = ri\xi r$, $f_2 = r\xi ri$, $f_3 = rir\xi$, $f_4 =  \xi i$,
and $\xi$ denotes the exotic map; again composition is left to right.  The map $g$ is also defined to interchange $2143$ and $3412$ and to fix $2413$ and $3142$.  Again it is readily checked that $g$ is consistent on the intersections of wedge classes and that it is an isomorphism, obviously an automorphism.  In this case the square of $g$ happens to be the symmetry $rc$.

The cardinality
$|\A^{(6)}_n|$ is the same as $|\A^{(5)}_n|$ except for $n=5$.

\section{Conclusion}
To within symmetry there are only 6 maximal isomorphisms and the classes $\A^{(i)}$ on which they are defined are rather restricted.  Non-maximal isomorphisms are (also to within symmetry) necessarily defined on subclasses of the $\A^{(i)}$ and are therefore also quite restricted although we have not attempted to categorize them.  Nevertheless some consequences for arbitrary isomorphisms are easily drawn. 
\begin{corollary}
If $\phi:\A\longrightarrow\B$ is an isomorphism and $\A$ has no basis element of length 5 or less then $\phi$ is induced by a symmetry of the pattern containment order.
\end{corollary}
\begin{proof}
If $\phi$ is not a symmetry then, to within a symmetry, $\A$ is contained within one of the classes $\A^{(i)}$, $i=2,3,4,5,6$ above and since each of these classes contains a basis element of length 5 or less so also will $\A$ itself.
\end{proof}

Furthermore, as all the proper classes that occur are subclasses of $\A^{(2)}$, whose growth rate is easily seen to be the reciprocal of the smallest positive root of the polynomial $p(x)$ we obtain:

\begin{corollary}
If $\A$ admits a non-symmetry isomorphism onto another class $\B$ then the growth rate of $\A$ is at most 5.90425.  Furthermore $\A$ contains only finitely many simple permutations and therefore is partially well-ordered.
\end{corollary}

Finally, we note that it also makes sense to pursue this investigation in search of classes having \emph{automorphisms} extending those of the finite class $\R$. Specifically, given a subgroup $G \leq \mathrm{Aut}(\R)$ there is a maximal class $\A^{(G)}$ containing $\R$ such that every element of $G$ extends to an automorphis of $\A^{(G)}$. The basic scheme of inductively building up the class still applies though one must ensure of course that \emph{all} the images of some permutation $\pi$ under the elements of $G$ can be defined. Because of this stringent condition, one would expect generally that the classes $\A^{(G)}$ should be relatively small. The group of automorphisms of $\R$ is isomorphic to $\left( \mathbb Z/2 \mathbb Z \right)^2 \times S_4$ since (and independently) the two monotone permutations in $\R$ can be fixed or interchanged, as can the two permutations of length four, while the remaining four non-monotone permutations of length three can be arbitrarily permuted. In searching for automorphisms.  we no longer have the option of applying different symmetries on each side to obtain an equivalent class -- we must apply the same symmetry. Effectively we have a single case for each orbit of subgroups of $\mathrm{Aut}(\R)$ under the action of the normal symmetries by conjugation. Rather than report complete results for this case, which we have not computed, we mention only one result:

\begin{theorem}
The maximal class $\A$ extending $\R$ with the property that every automorphism of $\R$ extends to $\A$ consists of $\R$ and all the permutations that can be expressed as the sum or skew sum of an increasing and a decreasing permutation. 
\end{theorem}

The basis of $\A$ consists of all the permutations of length 4 that do not belong to it, and $\A$ contains $4n-6$ permutations of each length $n \geq 5$ (containing 1, 2, 6 and 12 permutations of lengths 1 through 4 respectively). 

%
\bibliographystyle{acm}
\bibliography{refs}

\def\cprime{$'$}
\begin{thebibliography}{1}

\bibitem{PermLab1.0}
{\sc Albert, M.}
\newblock Permlab: Software for permutation patterns.
\newblock \url{http://www.cs.otago.ac.nz/PermLab}, 2012.

\bibitem{simion:restricted-perm:}
{\sc Simion, R., and Schmidt, F.~W.}
\newblock Restricted permutations.
\newblock {\em European J. Combin. 6}, 4 (1985), 383--406.

\bibitem{smith:permutation-rec:}
{\sc Smith, R.}
\newblock Permutation reconstruction.
\newblock {\em Electron. J. Combin. 13\/} (2006), Note 11, 8 pp.

\end{thebibliography}

\end{document}